\documentclass{article}

\usepackage{amsmath,amsthm,amssymb,tikz,bbm,float}

\newtheorem{theorem}{Theorem}

\newtheorem{lemma}{Lemma}
\newtheorem{proposition}{Proposition}
\newtheorem{conjecture}{Conjecture}
\theoremstyle{definition}

\newtheorem{example}{Example}

\title{Irrational braided generalized near-groups}
\author{Andrew Schopieray}
\date{}

\begin{document}

\maketitle
\begin{abstract}
We classify braided generalized near-group fusion categories whose global dimension is not an integer; there are exactly two up to Grothendieck equivalence and taking products with braided pointed fusion categories.
\end{abstract}

\section{Introduction}

The prototypical examples of fusion categories, in the sense of \cite{ENO}, are the categories of finite dimensional representations of finite groups.  Fusion categories of integer Frobenius-Perron dimension,  like these, arise as extensions of representation categories of finite dimensional quasi-Hopf algebras \cite[Theorem 8.33]{ENO}\cite[Proposition 1.8]{2019arXiv191212260G}.  Therefore, it is reasonable to consider these as classical objects.  Fusion categories with irrational Frobenius-Perron dimension are more modern creations; the first arose from the mathematical physics literature in the 1980's and can be constructed from representation categories of quantum groups at roots of unity \cite{MR4079742}.  These nonclassical examples have commutativity structures, called braidings, which make them more amenable to study than other families of potentially noncommutative fusion categories.

\par Here we continue the study of braided fusion categories focusing on those with elementary fusion rules.  The most elementary are the pointed fusion categories, or those whose simple objects are invertible.  These are classified by pre-metric groups \cite{MR1250465}, i.e.\ pairings of a finite abelian group $G$ and a quadratic form $q$.  The next level of complexity would be to assume there exists at least one noninvertible simple object \cite{220807319}.  When the global dimension is irrational, the invertible objects act transitively on the noninvertible simple objects \cite[Lemma 2.6]{220807319}.  Fusion categories with this transitivity property have been called generalized near-group fusion categories \cite{MR4167662,thornton2012generalized}.   Using number-theoretical constraints on global dimensions from \cite{MR3943751} and results on two-orbit fusion rings developed in \cite{220807319}, we prove in Theorem \ref{thm:irr} that there are exactly two Grothendieck equivalence classes of braided generalized near-group fusion categories with irrational global dimension up to taking products with braided pointed fusion categories.  This generalizes the known results for irrational braided generalized near-group fusion categories when the braiding is assumed to be nondegenerate \cite[Theorem IV.5.2]{thornton2012generalized} or slightly nondegenerate \cite[Theorem 5.8]{MR4167662}, but our reasoning does not rely on either of these results.

\par There are numerous reasons to seek classification results for braided generalized near-group fusion categories.  From the motivation of classical group theory, this classification contains as a subset a description of all finite groups $G$ such that the linear characters (degree 1) act on the nonlinear characters (degree larger than 1) transitively.  When the number of nonlinear characters is small \cite{palfy,MR222160}, this classification already contains multiple infinite families of finite groups and several interesting isolated examples.  From the motivation of representation theory, braided fusion categories give rise to braid group representations; for rigorous details and history we refer the reader to the introduction of \cite{MR2375313} and references therein.  Roughly, one can associate to any object $X$ in a braided fusion category and positive integer $n$, a representation of the braid group $B_n$ acting on $X^{\otimes n}$.  In all known examples when the Frobenius-Perron dimension of a braided fusion category is an integer, the images of these braid group representations factor over finite groups, and it has been conjectured that these two properties of braided fusion categories are equivalent \cite{MR2832261}.  Families of braided fusion categories which straddle the line between integer and non-integer Frobenius-Perron dimension, for example braided generalized near-group fusion categories, serve as a litmus test for such conjectures.  Furthermore, there is an inherent motivation to classify braided fusion categories as a class of algebraic objects in and of themselves, and as a subset of the larger class of fusion categories.  For example, for each positive integer $n$, it is known that there are finitely many braided fusion categories of rank $n$ \cite[Theorem 11]{MR4309554}; the corresponding statement for fusion categories is often conjectured, but is not proven.  Unsurprisingly, many braided fusion categories of small rank have elementary fusion rules and moreover have either group-like or generalized near-group fusion.

\par As an epilogue, we briefly discuss the integer Frobenius-Perron dimension case.  A large family of braided generalized near-group fusion categories with integer Frobenius-Perron dimension are the (non-generalized) near-group type, i.e.\ there exists exactly one noninvertible isomorphism class of simple objects.  These consist of the nilpotent examples, which are entirely described by the braided Tambara-Yamagami fusion categories \cite{siehler}, the symmetrically braided fusion categories $\mathrm{Rep}(F_q)$ where $F_q$ is the Frobenius group $\mathbb{F}_q\rtimes\mathbb{F}_q^\times$ for some prime power $q=p^n$, and the non-symmetrically braided fusion categories with the fusion rules of $\mathrm{Rep}(S_3)$ and $\mathrm{Rep}(A_4)$.  Along with the rank 2 braided fusion categories with irrational global dimension \cite{ostrik}, this is a complete classification of braided fusion categories with a unique non-invertible simple object \cite[Theorem III.4.6]{thornton2012generalized}.  In \cite{MR4195420}, an infinite family of nilpotent braided generalized near-group categories was constructed containing cyclic extensions of braided pointed fusion categories of rank 2.  These were named $N$-Ising braided fusion categories and represent a large class of slightly degenerate braided generalized near-group fusion categories \cite[Theorem 5.8]{MR4167662}.  This construction is readily generalizable (Example \ref{ty}) to any Tambara-Yamagami braided fusion category.  As a complement to the characterization of this family in terms of equivariantization \cite[Theorem 6.3]{MR4167662}, we prove in Proposition \ref{elem} that any nilpotent braided generalized near-group category is a graded extension of a braided pointed fusion category on an elementary abelian 2-group.  These categories have been called braided generalized Tambara-Yamagami categories \cite[Section 5]{MR3071133}\cite[Section 6]{MR4167662}, and we conjecture on their fusion rules (Conjecture \ref{con}) to stimulate future research.


\section{Preliminaries}\label{sec:def1}

Our main objects of study are fusion rings and their categorifications, known as fusion categories.  The following definitions and concepts are standard, and can be found in a textbook such as \cite{tcat} which we will cite often.

\subsection{Fusion rings}

Fusion rings are a natural generalization of integral group rings $\mathbb{Z}G$ for finite groups $G$, where invertibility has been replaced with a suitable notion of duality.  Although interesting algebraic objects themselves, their study is vital as they describe many of the combinatorial and number-theoretical aspects of fusion categories (Section \ref{sec:def2}).

\paragraph{Definition.}  A unital $\mathbb{Z}_{\geq0}$-ring $(R,B)$ is a pairing of an associative unital ring $R$ which is free as a $\mathbb{Z}$-module, and a distinguished basis $B=\{b_i:i\in I\}$ including the unit $b_0:=1_R$ such that for all $i,j\in I$, $b_ib_j=\sum_{k\in I}c_{i,j}^kb_k$ with $c_{i,j}^k\in\mathbb{Z}_{\geq0}$.  Such a ring is based if there exists an involution of $I$, denoted $i\mapsto i^\ast$, such that the induced map defined by
\begin{equation}
a=\sum_{i\in I}a_ib_i\mapsto a^\ast=\sum_{i\in I}a_ib_{i^\ast},\,\,\, a_i\in\mathbb{Z}
\end{equation}
is an anti-involution of $R$, and $c_{i,j}^0=1$ if $i=j^\ast$ and $c_{i,j}^0=0$ otherwise.  We say that $x,y\in R$ are dual to one another when $x=y^\ast$ and $x$ is self-dual when $x=x^\ast$.  A \emph{fusion ring} is a unital based $\mathbb{Z}_{\geq0}$-ring $(R,B)$ such that $|B|$ is finite.

\begin{example}[Near-group fusion rings]\label{near}
Let $G$ be a finite group and $\ell\in\mathbb{Z}_{\geq0}$.  The \emph{near-group} fusion ring $R(G,\ell)$ has basis $B=\{g:g\in G\}\cup\{\rho\}$ whose fusion rules are the group operation for elements of $G$, $g\rho=\rho g=\rho$ for all $g\in G$, and
\begin{equation}
\rho^2=\ell\rho+\sum_{g\in G}g.
\end{equation}
Examples of near-group fusion rings include the character rings of extra-special 2-groups, or more generally the character ring of any finite group with a unique irreducible character of degree greater than 1 \cite{MR222160}.
\end{example}

\paragraph{Gradings and Frobenius-Perron dimension} Let $(R,B)$ be a fusion ring.  We say that $R$ is graded by a finite group $K$ or that $R$ is $K$-graded if there is a partition $B=\cup_{g\in K}B_g$ into disjoint subsets $B_g\subset B$ such that for all $x\in B_g$ and $y\in B_h$, $xy$ is a $\mathbb{Z}_{\geq0}$-linear combination of elements of $B_{gh}$.  We denote the $\mathbb{Z}$-linear span of each $B_g$ by $R_g$ and refer to these sets as components of the $K$-grading.  A $K$-grading of a fusion ring $(R,B)$ is faithful if $B_g$ is non-empty for all $g\in K$.  Every fusion ring $(R,B)$ has a distinguished universal grading $U(R)$ which is faithful, and satisfies the universal property that any other faithful grading of $R$ by a finite group $K$ is determined by a surjective group homomorphism $U(R)\to K$ \cite[Corollary 3.6.6]{tcat}.  For example, every fusion ring has a canonical dimensional grading \cite[Section 1.2]{2019arXiv191212260G} by an elementary abelian $2$-group.  The dimensional grading was known previously in the specific case of fusion rings with $\mathrm{FPdim}(R)\in\mathbb{Z}$, in which case $\mathrm{FPdim}(x)^2\in\mathbb{Z}$ for all $x\in B$.  The trivial component $R_e$ of any grading is distinguished in that it forms a fusion subring of $R$, generated by $xx^\ast$ for $x\in B$.  We denote the trivial component of the universal grading by $R_\mathrm{ad}$ and refer to it as the adjoint subring.

\par Every fusion ring $(R,B)$ has a unique character $\mathrm{FPdim}:R\to\mathbb{C}$ with the property that $\mathrm{FPdim}(x)>0$ for all $x\in B$ \cite[Proposition 3.3.6(3)]{tcat}.  These Frobenius-Perron dimensions $\mathrm{FPdim}(x)$ for $x\in B$ can be computed as the maximal eigenvalues of the matrices of (left) multiplication by $x$, and then one can extend this function linearly to all of $R$.   If $\mathrm{FPdim}(x)=1$ for $x\in R$, then $x\in B$ and we call $x$ invertible.  The invertible elements of $R$ form a finite group $G_R$ which acts on $B$ by permutation.  If $G_R=R$, we say $R$ is pointed.

\begin{example}[Generalized near-group fusion rings]\label{gen}
When $B$ has two orbits under the action of $G_R$, which we will denote $|B/G_R|=2$, then there exists $d\in\mathbb{R}_{>1}$ such that $\mathrm{FPdim}(x)\in\{1,d\}$ for all $x\in B$.  Number-theoretic aspects of these \emph{generalized near-group} fusion rings are simplified since if $\mathrm{FPdim}(x)=d$ for some $b\in B$, then
\begin{equation}\label{eq:1}
d^2-rd-|H_R|=0
\end{equation}

where $H_R\trianglelefteq G_R$ is the normal subgroup stabilizing any basis element of Frobenius-Perron dimension $d$ \cite[Lemma 2.8(e)]{220807319}, and $r\in\mathbb{Z}_{\geq0}$ is the multiplicity of basis elements of Frobenius-Perron dimension $d$ appearing as summands of $xx^\ast$.  If $R$ is any near-group fusion ring (Example \ref{near}) and $K$ any finite group, then $R\times\mathbb{Z}K$ is a generalized near-group fusion ring.  There are multitudes of examples which are not of this factored form (Example \ref{ty}).
\end{example}

\par The following lemma was first used in \cite[Chapter IV]{thornton2012generalized} as a statement about categories, but we include its proof here for completeness, and as more general statement about generalized near-group fusion rings.

\begin{lemma}\label{bigone}
Let $(R,B)$ be a fusion ring with $|B/G_R|=2$.  Then $(R_\mathrm{ad})_\mathrm{ad}=R_\mathrm{ad}$ or $(R_\mathrm{ad})_\mathrm{ad}$ is trivial.
\end{lemma}

\begin{proof}
Recall that $R_\mathrm{ad}$ is generated by the basis elements appearing as summands of $xx^\ast$ across all $x\in B$ \cite[Definition 3.6.1]{tcat}, but since $xx^\ast=yy^\ast$ for all noninvertible $x,y\in B$ \cite[Lemma 2.8(d)]{220807319}, if there exists noninvertible $x\in R_\mathrm{ad}$, then $R_\mathrm{ad}\subset (R_\mathrm{ad})_\mathrm{ad}$ and moreover $(R_\mathrm{ad})_\mathrm{ad}=R_\mathrm{ad}$.  Otherwise all basis elements of $R_\mathrm{ad}$ are invertible, and thus $(R_\mathrm{ad})_\mathrm{ad}$ is trivial. 
\end{proof}

Lemma \ref{bigone} demonstrates that there is a very natural stratification of generalized near-group fusion rings into those which are nilpotent in the sense of \cite{nilgelaki}, and those which are extensions of self-adjoint generalized near-group fusion rings by finite groups.

\begin{example}[Nilpotent generalized near-group fusion rings]\label{ty}
\par The following nilpotent fusion rings are generalized from those defined in \cite{MR4195420}.  Let $m,n\in\mathbb{Z}_{\geq1}$.  Denote by $R(m,n)\subset\mathbb{Z}C_{2^m}\times R(C_2^n,0)$ the fusion subring generated by $g\times\rho$ where $g$ is any generator of $C_{2^m}$ and $\rho$ is the unique noninvertible basis element of $R(C_2^n,0)$ (see Example \ref{gen}).  The $2^{m-1}$ noninvertible basis elements of $R(m,n)$ are $g^j\times\rho$ for all odd $j\in\mathbb{Z}$, and the $2^{m+n-1}$ invertible elements are $g^j\times h$ for all $h\in C_2^n$ and even $j\in\mathbb{Z}$.  In particular, $R(m,n)_\mathrm{pt}=\mathbb{Z}(C_{2^{m-1}}\times C_2^n)$.  We later conjecture (Conjecture \ref{con}) that up to products with pointed fusion rings, any braided generalized Tambara-Yamagami fusion category has fusion rules $R(m,n)$ for some $m,n\in\mathbb{Z}_{\geq1}$. 
\end{example}


\subsection{Fusion categories and braidings}\label{sec:def2}

Fusion categories are the categorical analog of fusion rings.  As such, the language and notation used in Section \ref{sec:def1} carry over naturally.  In this section we will focus on the categorical aspects that allow a proof of Theorem \ref{thm:irr}.

\paragraph{Definitions.}  \emph{Fusion categories} (over $\mathbb{C}$) are $\mathbb{C}$-linear semisimple rigid monoidal categories $\mathcal{C}$ (with product $\otimes$, monoidal unit $\mathbbm{1}_\mathcal{C}$, and duality $X\mapsto X^\ast$) whose set of isomorphism classes of simple objects, $\mathcal{O}(\mathcal{C})$, is finite and includes $\mathbbm{1}_\mathcal{C}$.  These assumptions ensure the Grothendieck ring of $\mathcal{C}$ is a fusion ring with basis $\mathcal{O}(\mathcal{C})$.  All of the structure morphisms in a fusion category are canonical \cite[Chapter 2]{tcat} except the associators, i.e.\ the natural isomorphisms $\alpha_{X,Y,Z}:(X\otimes Y)\otimes Z\to X\otimes(Y\otimes Z)$ for all $X,Y,Z\in\mathcal{O}(\mathcal{C})$ which must satisfy the ``pentagon axioms'' \cite[Definition 2.1.1]{tcat}.  In particular, the tensor unit $\mathbbm{1}_\mathcal{C}$ and its structure morphisms are unique up to a unique isomorphism \cite[Proposition 2.2.6]{tcat}, and the same is true for $X^\ast$ for any $X\in\mathcal{C}$ \cite[Proposition 2.10.5]{tcat}.  It is therefore productive and equivalent to think of fusion categories as fusion rings equipped with solutions to the system of algebraic equations dictated by the pentagon axioms.  With this view, it is clear that for each fusion category $\mathcal{C}$ and $\sigma\in\mathrm{Gal}(\overline{\mathbb{Q}}/\mathbb{Q})$, where $\overline{\mathbb{Q}}$ is the algebraic closure of $\mathbb{Q}$, there exists a Galois conjugate fusion category $\mathcal{C}^\sigma$ whose associators are defined by applying $\sigma$ to the solutions of the pentagon equations for $\mathcal{C}$ \cite[Sections 4.2--4.3]{https://doi.org/10.48550/arxiv.1305.2229}.

\begin{example}[Irrational generalized near-group fusion categories]\label{yurrr}  Recall from Example \ref{gen} that the unique nontrivial Frobenius-Perron dimension $d$ of basis elements in a generalized near-group fusion ring $R$, satisfies $d^2-rd-|H_R|=0$ for some $r\in\mathbb{Z}_{\geq0}$.    There is no restriction on the $r$ which can occur on the level of fusion rings, but it was shown in \cite[Proposition 3.3]{220807319} that if $R$ is the Grothendieck ring of a fusion category and $d\not\in\mathbb{Z}$, then $|H_R|$ divides $r$, i.e.\ there exists $k\in\mathbb{Z}_{\geq0}$ such that $d^2-k|H_R|d-|H_R|=0$.  As a result, $\mathrm{FPdim}(R)=|G_R|(2+kd)$.  This fact will be used in the proofs of Lemmas \ref{tan}--\ref{nottan}.
\end{example}

\par One of the beautiful aspects of fusion categories is that they can be structurally commutative in a variety of ways.  Rather than having $xy=yx$ for every $x,y$ in a fusion ring $R$, if one has natural isomorphisms $\tau_{X,Y}:X\otimes Y\to Y\otimes X$ for all $X,Y$ satisfying compatibilities \cite[Definition 8.1.1]{tcat} with the associative isomorphisms in a fusion category $\mathcal{C}$, we say that $\mathcal{C}$ equipped with these commutator isomorphisms forms a \emph{braided fusion category}.  Paramount to the study of braided fusion categories is the notion of centralizer fusion subcategories.  Given a fusion subcategory $\mathcal{D}$ of a braided fusion category $\mathcal{C}$, the centralizer of $\mathcal{D}$ in $\mathcal{C}$, denoted $C_\mathcal{C}(\mathcal{D})$ is the full fusion subcategory of $\mathcal{C}$ whose objects are
\begin{equation}
\{X\in\mathcal{C}:\tau_{Y,X}\tau_{X,Y}=\mathrm{id}_{X\otimes Y}\text{ for all }Y\in\mathcal{C}\}.
\end{equation}
At one extreme, we note $C_\mathcal{C}(\mathbbm{1})=\mathcal{C}$ where we briefly abuse notation by using $\mathbbm{1}$ to refer to the fusion subcategory generated by the tensor unit.  The opposite extreme, $C_\mathcal{C}(\mathcal{C})$, is known as the symmetric center of $\mathcal{C}$.  The symmetric center is one of the defining characteristics of braided fusion categories.  In particular, when $C_\mathcal{C}(\mathcal{C})=\mathbbm{1}$ we refer to $\mathcal{C}$ as nondegenerately braided.

\begin{example}[Symmetrically braided and Tannakian fusion categories]\label{ex:sym}  If $\mathcal{C}$ is a braided fusion category and $C_\mathcal{C}(\mathcal{C})=\mathcal{C}$, then we say $\mathcal{C}$ is \emph{symmetrically braided}.  It has been shown \cite[Theorem 9.9.22]{tcat} that any symmetrically braided fusion category is braided equivalent to a braided fusion category $\mathrm{Rep}(G,z)$ for a finite group $G$ and involution $z\in G$, which is equivalent to $\mathrm{Rep}(G)$ as a fusion category, with the braiding twisted by $z$ \cite[Example 9.9.1(3)]{tcat}.  In particular, when a symmetrically braided fusion category $\mathcal{C}$ is braided equivalent to $\mathrm{Rep}(G,z)$ with $z$ trivial, then $\mathcal{D}\simeq\mathrm{Rep}(G)$ and we say $\mathcal{D}$ is  \emph{Tannakian}.  Otherwise $\mathcal{D}$ is \emph{super-Tannakian}.  Every symmetrically braided fusion category $\mathcal{D}$ has a canonical $C_2$-grading (possibly non-faithful) whose trivial component is a Tannakian fusion subcategory \cite[Corollary 9.9.32]{tcat}.
\end{example}

One final structure that will be needed in what follows is based on the relation between the Frobenius-Perron and global dimensions \cite[Section 2.2]{ENO} of a fusion category.  In particular, any spherical structure \cite[Definition 4.7.14]{tcat} on a fusion category $\mathcal{C}$ provides a categorical dimension character, denoted $X\mapsto\dim(X)$ for $X\in\mathcal{C}$, of the Grothendieck ring of $\mathcal{C}$ which may or may not coincide with the Frobenius-Perron dimension character (see Section \ref{sec:def1}).  When $\mathrm{FPdim}(\mathcal{C})=\dim(\mathcal{C}):=\sum_{X\in\mathcal{O}(\mathcal{C})}\dim(X)^2$, we say $\mathcal{C}$ is pseudounitary, and in this case there exists a canonical spherical structure so that $\dim=\mathrm{FPdim}$ as characters \cite[Proposition 9.5.1]{tcat}.  Pragmatically, this allows us to never mention a specific spherical structure and allows results about categorical and Frobenius-Perron dimensions to be used interchangeably.  By \cite[Lemma 3.1]{220807319}, all of the categories considered in Sections \ref{sec:yurrr}--\ref{sec:nurrr} are Galois conjugate to pseudounitary fusion categories.

\paragraph{De-equivariantization.} The following construction can be considered for arbitrary fusion categories \cite[Sections 4.15 \& 8.23]{tcat}, but we will describe it only in the context of braided fusion categories.  Given a braided fusion category $\mathcal{C}$, and a Tannakian fusion subcategory $\mathrm{Rep}(G)\simeq\mathcal{D}\subset\mathcal{C}$, one may consider the de-equivariantization category $\mathcal{C}_G$ \cite[Section 8.23]{tcat}.  The de-equivariantization $\mathcal{C}_G$ is a fusion category, and is braided when $\mathcal{D}\subset C_\mathcal{C}(\mathcal{C})$ \cite[Proposition 8.23.8]{tcat}.  On the level of global dimension, $\dim(\mathcal{C}_G)|G|=\dim(\mathcal{C})$ \cite[Proposition 4.26]{DGNO}.  The process of de-equivariantization is analogous to quotienting a finite group by a normal subgroup.  Precisely, using the notation of \cite[Section 8.4]{tcat}, if $\mathcal{C}(G,q)$ is a pre-metric group and $H\subset G$ a subgroup, then $\mathcal{C}(G,q)_H\simeq\mathcal{C}(G/H,\tilde{q})$ for a particular quadratic form $\tilde{q}$ on $G/H$ \cite[Example 8.23.10]{tcat}.  This construction on fusion categories is useful for the same reason normal subgroups are useful in the study of finite groups: the quotienting process reduces arguments to understanding small examples and how the original examples can be recovered from attempting to reverse the quotienting process.  Equivariantization, denoted $\mathcal{C}^G$, of a braided fusion category $\mathcal{C}$ by a finite group $G$ is the reverse of de-equivariantization in the sense that there are canonical equivalences $(\mathcal{C}^G)_G\simeq\mathcal{C}$ and $(\mathcal{C}_G)^G\simeq\mathcal{C}$ \cite[Section 4.2]{DGNO}.

\par Another effective way of thinking about de-equivariantization is in terms of modules over commutative algebra objects in braided fusion categories \cite[Section 8.8]{tcat}.  We will not expound on this viewpoint in detail other than to make an observation about dimensions of simple objects in de-equivariantizations of braided fusion categories.  In particular, every Tannakian fusion category $\mathrm{Rep}(G)$ has a canonical commutative algebra object $A$: the algebra of functions on $G$ \cite[Example 8.8.9]{tcat}.  If $\mathcal{C}$ is a braided fusion category containing a subcategory braided equivalent to $\mathrm{Rep}(G)$, then the de-equivariantization $\mathcal{C}_G$ is the category of (left) $A$-module objects in $\mathcal{C}$ \cite[Example 3.8]{DMNO}.  The tensor identity in $\mathcal{C}_G$ is the algebra object $A$ itself, and all other simple $A$-modules are summands of the free $A$-modules $A\otimes X$ for $X\in\mathcal{C}$.  The free $A$-modules $A\otimes X$ when $X$ is invertible are all simple by \cite[Lemma 3.2]{Ostrik2003} and invertible, and when $X$ is not invertible, any simple summand $Y$ of the free module $A\otimes X$ has a dimension which is a rational multiple of $\dim(X)$ \cite[Theorem 1.18]{KiO}.


\section{Irrational generalized near-groups}\label{sec:yurrr}


If $\mathcal{C}$ is a generalized near-group fusion category and $\mathrm{FPdim}(\mathcal{C})\not\in\mathbb{Z}$, then $\mathcal{C}$ is not nilpotent \cite[Remark 4.7(6)]{nilgelaki}.  In this case Lemma \ref{bigone} implies that $\mathcal{C}$ is an extension of a generalized near-group fusion category with $\mathcal{C}=\mathcal{C}_\mathrm{ad}$, and so this will be a natural assumption moving forward.  Assume further that $\mathcal{C}$ is braided.  Then $C_\mathcal{C}(\mathcal{C})=\mathcal{C}_\mathrm{pt}$ is symmetrically braided from \cite[Proposition 3.13]{MR4167662} and the assumption $\mathcal{C}=\mathcal{C}_\mathrm{ad}$.   Symmetrically braided fusion categories are Tannakian or super-Tannakian (see Example \ref{ex:sym}), i.e.\ they possess a maximal Tannakian fusion subcategory with one-half the dimension of the original \cite[Corollary 9.9.32]{tcat}, so we proceed by these cases.  The end result is Theorem \ref{thm:irr} which states that there are exactly two generalized near-group braided fusion categories up to Grothendieck equivalence and taking products with braided pointed fusion categories.  The braided categorifications of these fusion rules have been previously classified as cited in the following example.

\begin{example}\label{lie}
A large number of braided fusion categories with irrational global dimension spawn from the representation theory of quantum groups at roots of unity.  The general construction has a long circuitous history, but the combinatorics and numerics of these categories are very tractable to anyone with a modest understanding of Lie theory \cite{MR4079742}.  To each complex finite-dimensional simple Lie algebra $\mathfrak{g}$ and root of unit $q$ such that $q^2$ is a primitive $\ell$th root of unity with $\ell\in\mathbb{Z}$ which is greater than or equal to the dual Coxeter number of $\mathfrak{g}$, there exists a spherical braided fusion category $\mathcal{C}(\mathfrak{g},\ell,q)$.  For our purposes, we will only consider the adjoint subcategories $\mathcal{C}(\mathfrak{sl}_2,\ell,q)_\mathrm{ad}$ for $\ell\geq5$.  It is well-known \cite{MR1237835} that any braided fusion category with the fusion rules of $\mathcal{C}(\mathfrak{sl}_2,\ell,q)_\mathrm{ad}$ for some $\ell$ and $q$ is furthermore braided equivalent to $\mathcal{C}(\mathfrak{sl}_2,\ell,q)_\mathrm{ad}$ for some choice of $\ell$ and $q$.  For example, all rank 2 braided fusion categories with irrational global dimension \cite[Section 2.5]{ostrik} are braided equivalent to $\mathcal{C}(\mathfrak{sl}_2,5,q)_\mathrm{ad}$ for one of the four choices of primitive 10th roots of unity $q$.
\end{example}

\begin{lemma}\label{tan}
Let $\mathcal{C}$ be a braided generalized near-group fusion category such that $\mathcal{C}=\mathcal{C}_\mathrm{ad}$ and $\mathrm{FPdim}(\mathcal{C})\not\in\mathbb{Z}$.  If $\mathcal{C}_\mathrm{pt}$ is Tannakian, then $\mathcal{C}$ is equivalent to $\mathcal{C}(\mathfrak{sl}_2,5,q)_\mathrm{ad}$ where $q$ is a primitive $10$th root of unity.
\end{lemma}

\begin{proof}
Assume the Grothendieck ring of $\mathcal{C}$ is a generalized near-group fusion ring $(R,B)$ with group of invertible elements $G:=G_R$ and fixed-point subgroup $H_R=:H\trianglelefteq G$ (see Example \ref{gen}).  Without loss of generality \cite[Lemma 3.1]{220807319} we may assume $\mathcal{C}$ is equipped with the unique spherical structure such that $\dim(\mathcal{C})=\mathrm{FPdim}(\mathcal{C})=|G|(2+kd)$ where $k\in\mathbb{Z}_{>0}$ is such that every noninvertible $X\in\mathcal{O}(\mathcal{C})$ has $d:=\mathrm{FPdim}(X)=\dim(X)$ which is the largest root of $x^2-k|H|x-|H|$ (see Example \ref{yurrr}).  We have assumed $C_\mathcal{C}(\mathcal{C})=\mathcal{C}_\mathrm{pt}\simeq\mathrm{Rep}(G)$ is a braided equivalence for a finite group $G$, so we can consider the de-equivariantization $\mathcal{C}_G$ (see Section \ref{sec:def2}) which is a fusion category with $\dim(\mathcal{C}_G)=\mathrm{FPdim}(\mathcal{C})=2+kd$.  Let $\sigma\in\mathrm{Gal}(\overline{\mathbb{Q}}/\mathbb{Q})$ be any automorphism such that $\sigma|_{\mathbb{Q}(d)}\in\mathrm{Gal}(\mathbb{Q}(d)/\mathbb{Q})$ is nontrivial, which exists since $d\not\in\mathbb{Z}$.  Noting $d\sigma(d)=-|H|$, and $d+\sigma(d)=k|H|$,
\begin{equation}
\dim((\mathcal{C}_G)^\sigma)=2+k\sigma(d)=2-\frac{k|H|}{d}=2-(d+\sigma(d))d=1-\frac{\sigma(d)}{d}.
\end{equation}
Thus by \cite[Theorem 1.1.2]{MR3943751}, we require $
1/3<-\sigma(d)/d=|H|/d^2=(1+kd)^{-1}$.  As $d>1$, we must have $k=1$, thus $d=(1/2)(|H|+\sqrt{|H|(|H|+4)})<2$, which is only true for $|H|=1$.

\par The dimensions of simple objects in $\mathcal{C}$ are now determined as $1$ or $d=(1/2)(1+\sqrt{5})$.  Assume $X\in\mathcal{O}(\mathcal{C})$ with $\dim(X)=d$.  Then $X\otimes X^\ast\cong\mathbbm{1}\oplus Y$ for some noninvertible $Y\in\mathcal{O}(\mathcal{C})$ as $d^2=1+d$ is the unique decomposition of $d^2$ into a sum of $1$ and $d$ with nonnegative integer coefficients, which implies $Y\cong Y^\ast$.  But by \cite[Lemma 2.8(d)]{220807319}, $Y\otimes Y^\ast\cong X\otimes X^\ast\cong\mathbbm{1}\oplus Y$.  Therefore $Y$ $\otimes$-generates the category $\mathcal{C}_\mathrm{ad}=\mathcal{C}$ for which $|\mathcal{O}(\mathcal{C})|=2$, which are classified as in the statement of the lemma (see Example \ref{lie}).
\end{proof}

\par In the remainder of our proof, we will use numerics which arise from realizing $\mathcal{C}$ as a pseudounitary pre-modular category \cite[Section 8.13]{tcat}, i.e.\ a braided fusion category equipped with a spherical structure.  In particular, for all $X,Y\in\mathcal{O}(\mathcal{C})$ there exist pre-modular data $s_{X,Y},\theta_X\in\mathbb{C}$ which are related by the balancing equation \cite[Proposition 8.13.8]{tcat}
\begin{equation}\label{nein}
s_{X,Y}=\theta_X^{-1}\theta_Y^{-1}\sum_{Z\in\mathcal{O}(\mathcal{C})}N_{X,Y}^Z\theta_Z\dim(Z)
\end{equation}
where $N_{X,Y}^Z:=\dim_\mathbb{C}(\mathrm{Hom}_\mathcal{C}(X\otimes Y,Z))$ are the fusion rules of $\mathcal{C}$.  This is convenient since pseudounitarity implies that $X,Y\in\mathcal{O}(\mathcal{C})$ centralize one another ($\tau_{Y,X}\tau_{X,Y}=\mathrm{id}_{X\otimes Y}$) if and only if $s_{X,Y}=\dim(X)\dim(Y)$ \cite[Proposition 2.5]{mug1}.  For example $X,Y\in\mathcal{O}(\mathcal{C}_\mathrm{pt})$ centralize one another if and only if $s_{X,Y}=1$.

\par The following is a generalization of \cite[Proposition III.3.10]{thornton2012generalized} to braided generalized near-group fusion categories. 

\begin{lemma}\label{tanlem}
Let $\mathcal{C}$ be a braided generalized near-group fusion category with fixed-point subgroup $H$.  If $\mathcal{C}=\mathcal{C}_\mathrm{ad}$ and $\mathcal{C}$ is not symmetrically braided, then $\theta_h=1$ for all $h\in H$, i.e.\   the pointed fusion subcategory of $\mathcal{C}$ corresponding to $H$ is braided equivalent to $\mathrm{Rep}(H)$.
\end{lemma}

\begin{proof}
Recall that our assumptions imply $C_\mathcal{C}(\mathcal{C})=\mathcal{C}_\mathrm{pt}$ \cite[Proposition 3.13]{MR4167662}.  Using (\ref{nein}), we compute for any noninvertible $X\in\mathcal{O}(\mathcal{C})$ and invertible $h\in H$,
\begin{equation}
\dim(X)=s_{h,X}=\theta_h^{-1}\theta_X^{-1}\theta_{h\otimes X}\dim(h\otimes X)=\theta_h^{-1}\dim(X),
\end{equation}
proving $\theta_h=1$ for all $h\in H$.  The structure of the pointed fusion subcategory of $\mathcal{C}$ corresponding to $H$ is then dictated by the pre-metric group $H$ with trivial quadratic form \cite[Theorem 8.4.12]{tcat}, which is the same pre-metric group corresponding to $\mathrm{Rep}(H)$.
\end{proof}

\begin{lemma}\label{biglem}
If $\mathcal{C}$ is a braided generalized near-group fusion category with a braided equivalence $\mathcal{C}_\mathrm{ad}\simeq\mathcal{C}(\mathfrak{sl}_2,8,q)$ for a primitive $16$th root of unity $q$, then $\mathcal{C}\simeq\mathcal{C}_\mathrm{ad}\boxtimes\mathcal{P}$ for a braided pointed fusion category $\mathcal{P}$.
\end{lemma}

\begin{proof}
\par Assume $\mathcal{C}$ is a $K$-extension of $\mathcal{C}_\mathrm{ad}$ for a finite group $K$, i.e.\ $K$ is the universal grading group of $\mathcal{C}$  (see Section \ref{sec:def1}).  Note that all noninvertible $X\in\mathcal{O}(\mathcal{C})$ have $\dim(X)=1+\sqrt{2}$, and $2\cdot1^2+2\cdot(1+\sqrt{2})^2$ is the unique decomposition of $\dim(\mathcal{C}_\mathrm{ad})=8+4\sqrt{2}$ into a nonnegative integer linear combination of $1^2$ and $(1+\sqrt{2})^2$.  Thus \cite[Theorem 3.5.2]{tcat} implies that each graded component of $\mathcal{C}$ contains two invertible and two noninvertible simple objects.  Denote the nontrivial invertible object of $\mathcal{C}_\mathrm{ad}$ by $\delta$ which has order 2 and necessarily permutes the invertible objects in each graded component.  We claim that $h^n\neq\delta$ for any $h\in G$ and $n>1$ which is sufficient to prove that $G\cong C_2\times L$ for some subgroup $L\subset G$ where $\delta$ generates the $C_2$ factor.  The group $G$ is abelian so we may assume $G\cong G_0\times G_1$ is a factorization into the subgroup $G_0$ of elements whose order is a power of $2$ and subgroup $G_1$ whose elements have odd-order.  Since $\delta\in G_0$, it suffices to assume there exists $g\in K$ of order $2$ and invertible $h\in\mathcal{C}_g$ with $h^2=\delta$, i.e.\ $\delta\otimes h=h^\ast=h^{-1}\neq h$ is the invertible object of $\mathcal{C}_g$ distinct from $h$.  Direct computation gives $\theta_\delta=-1$ \cite[Example 18]{MR4079742} for any choice of primitive 16th root of unity $q$ such that $\mathcal{C}_\mathrm{ad}\simeq\mathcal{C}(\mathfrak{sl}_2,8,q)_\mathrm{ad}$.  One consequence is that $\delta\not\in H$ by Lemma \ref{tanlem}.  But on one hand, $\theta_{h^\ast}=\theta_h$ from the definition of $\theta$ \cite[Definition 8.10.1]{tcat} and sphericality, and on the other hand $1=s_{\delta,h}=-\theta_h^{-1}\theta_{h^\ast}$ by Equation (\ref{nein}) and the fact that $\mathcal{C}_\mathrm{ad}\subset C_\mathcal{C}(\mathcal{C}_\mathrm{pt})$ \cite[Proposition 2.1]{MR4195420}.  These cannot occur simultaneously so we conclude $h^2=e$ for both invertible $h\in\mathcal{C}_g$.  Moreover $G\cong C_2\times L$ where $\delta$ generates the $C_2$ factor.  Let $\mathcal{P}\subset\mathcal{C}$ be the pointed braided fusion subcategory corresponding to $L\subset G$.  Then every $X\in\mathcal{O}(\mathcal{C})$ has a unique factorization $Y\otimes Z$, where $Y\in\mathcal{O}(\mathcal{C}_\mathrm{ad})$ and $Z\in\mathcal{O}(\mathcal{P})$.  Hence $\mathcal{C}\simeq\mathcal{C}_\mathrm{ad}\boxtimes\mathcal{P}$ \cite[Corollary 3.9]{MR3633318} since $\mathcal{C}$ is braided.
\end{proof}

\begin{lemma}\label{nottan}
Let $\mathcal{C}$ be a braided generalized near-group fusion category such that $\mathcal{C}=\mathcal{C}_\mathrm{ad}$ and $\mathrm{FPdim}(\mathcal{C})\not\in\mathbb{Z}$.  If $\mathcal{C}_\mathrm{pt}$ is super-Tannakian, then $\mathcal{C}$ is equivalent to $\mathcal{C}(\mathfrak{sl}_2,8,q)_\mathrm{ad}$ where $q$ is any primitive $16$th root of unity.
\end{lemma}

\begin{proof}
We retain the same notation as in the proof of Lemma \ref{tan}.  Let $\mathcal{D}\subset\mathcal{C}_\mathrm{pt}=C_\mathcal{C}(\mathcal{C})$ be a maximal Tannakian subcategory with $\mathrm{FPdim}(\mathcal{C}_\mathrm{pt})=2\mathrm{FPdim}(\mathcal{D})$ \cite[Corollary 9.9.32(ii)]{tcat}. Thus $\mathcal{D}\simeq\mathrm{Rep}(K)$ is a braided equivalence for a finite group $K\subset G$ of index 2.  There are exactly 2 orbits of the $\otimes$-action of $\mathcal{O}(\mathcal{D})$ on $\mathcal{O}(\mathcal{C}_\mathrm{pt})$ which implies $|\mathcal{O}((\mathcal{C}_K)_\mathrm{pt})|=2$.  Fix a noninvertible simple object $X$ of $\mathcal{C}_K$ which must have $\dim(X)=p_Xd$ for some $p_X\in\mathbb{Q}_{>0}$ \cite[Corollary 2.13]{MR3059899} (also refer to Section \ref{sec:def2}).  Decomposing $X\otimes X^\ast$ into simple summands and measuring its dimension implies
\begin{equation}\label{siiix}
(p_Xd)^2=q_Xd+r_X
\end{equation}
for some $q_X\in\mathbb{Q}_{>0}$ where $r_X=2$ if $X$ is fixed by the nontrivial invertible object and $r_X=1$ otherwise since this is the only nontrivial simple object of rational dimension in $\mathcal{C}_K$.  Using $d^2=k|H|d+|H|$, we then have
\begin{equation}
(p_X^2k|H|-q_X)d+(p_X^2|H|-r_X)=0.
\end{equation} 
As $d\not\in\mathbb{Z}$, then $p_X^2|H|=r_X$ and $p_X^2k|H|=q_X$; moreover $p_X^2=r_X/|H|$ and $q_X=kr_X$.  This implies $r_X=r_Y$ for all all noninvertible $X,Y\in\mathcal{O}(\mathcal{C}_K)$ as $1/|H|$ and $2/|H|$ cannot both be perfect squares in $\mathbb{Q}$.   Therefore $p_X$ is determined by $|H|$, and moreover $q_X$ is determined by $k$ and $|H|$.  In other words, all noninvertible simple objects of $\mathcal{C}_K$ are the same irrational dimension, i.e.\ $\mathcal{C}_K$ is a braided generalized near-group fusion category with a rank 2 pointed subcategory and $\mathrm{FPdim}(\mathcal{C}_K)\not\in\mathbb{Z}$.  Such a braided fusion category has $|\mathcal{O}(\mathcal{C}_K)|\leq4$ since the invertible objects act transitively on the noninvertible simple objects.  All braided fusion categories fitting this description are known \cite{MR3548123,ost15}.  We conclude the $\dim(\mathcal{C}_K)=(1/2)(5+\sqrt{5})$ or $\dim(\mathcal{C}_K)=4(2+\sqrt{2})$.  In either case $r_X=1$ in Equation (\ref{siiix}) and thus $q_X=k$.  But 
\begin{equation}\label{ate}
\dim(\mathcal{C}_K)=\dfrac{\dim(\mathcal{C})}{\dim(\mathcal{C}_\mathrm{pt})/2}=2(2+kd)
\end{equation}
as well, thus $\dim(\mathcal{C}_K)\neq(1/2)(5+\sqrt{5})$ since $\dim((\mathcal{C}_K)_\mathrm{pt})=2\nmid(1/2)(5+\sqrt{5})$ violating \cite[Proposition 8.15]{ENO}.  Therefore $\mathcal{C}_K\simeq\mathcal{C}(\mathfrak{sl}_2,8,q)$ for a primitive $16$th root of unity $q$ and $p_Xd=1+\sqrt{2}$.  Equation (\ref{ate}) implies $2(2+2d)=2(2+kd)=4(2+\sqrt{2})$, thus $kd=2+2\sqrt{2}$.  Noting that the algebraic norm of $2+2\sqrt{2}$ is $4$ and $k\in\mathbb{Z}_{\geq1}$ divides it as $d$ is an algebraic integer, then $k\in\{1,2\}$.

\paragraph{Case $k=2$:} Here we have $d=1+\sqrt{2}$, and we note that $1\cdot1+2\cdot(1+\sqrt{2})$ is the unique decomposition of $(1+\sqrt{2})^2$ into a nonnegative integer linear combination of $1$ and $1+\sqrt{2}$.  Therefore for any $X\in\mathcal{O}(\mathcal{C})$ with $\dim(X)=d$, we have $X\otimes X^\ast\cong\mathbbm{1}\oplus Y\oplus Z$ for some $Y,Z\in\mathcal{O}(\mathcal{C})$ with $\dim(Y)=\dim(Z)=d$.  Hence $Y\otimes Y^\ast\cong\mathbbm{1}\oplus Y\oplus Z$ and $Z\otimes Z^\ast\cong\mathbbm{1}\oplus Y\oplus Z$ as well.  Note that $Y\not\cong Z$, or else $Y$ $\otimes$-generates a braided near-group fusion subcategory of rank $3$; but this cannot occur as all such braided fusion categories have integer dimension \cite[Theorem III.4.6]{thornton2012generalized}.  These fusion rules imply \cite[Proposition 2.10.8]{tcat}
\begin{align}
1&=\dim_\mathbb{C}(\mathrm{Hom}_\mathcal{C}(Y\otimes Y^\ast,Z))=\dim_\mathbb{C}(\mathrm{Hom}_\mathcal{C}(Y,Y\otimes Z)),\mathrm{ and} \\
1&=\dim_\mathbb{C}(\mathrm{Hom}_\mathcal{C}(Z\otimes Z^\ast,Y))=\dim_\mathbb{C}(\mathrm{Hom}_\mathcal{C}(Z,Y\otimes Z)).
\end{align}
Therefore, there exists a nontrivial invertible object $\delta$ of order 2 such that $Y\otimes Z\cong\delta\oplus Y\oplus Z$.  Now the left-hand side of $Y\otimes Y^\ast\cong\mathbbm{1}\oplus Y\oplus Z$ being self-dual implies $Y$ is self-dual, or $Y^\ast\cong Z$.  In either case, $\delta\otimes Y\cong Z^\ast$ and $\delta\otimes Z\cong Y^\ast$ are isomorphic to either $Y$ or $Z$.  Hence $|\mathcal{O}(\mathcal{C})|=|\mathcal{O}(\mathcal{C}_\mathrm{ad})|=4$ and $\mathcal{C}$ is a braided fusion category with $\dim(\mathcal{C})=4(2+\sqrt{2})$, which are classified as in the statement of the lemma \cite[Theorem 4.11]{MR3548123}.

\paragraph{Case $k=1$:} Here we have $\dim(X)=d=2+2\sqrt{2}$ (a root of $x^2-4x-4$) for some $X\in\mathcal{O}(\mathcal{C})$, hence $|H|=4$.  Let $h\in H$ of order 2.  Since $H$ is Tannakian by Lemma \ref{tanlem}, we can de-equivariantize by $L$, the order 2 subgroup generated by $h$.  As fixed-points of elements of order $2$, any $X\in\mathcal{O}(\mathcal{C})$ of dimension $d$ splits into two non-isomorphic simple objects of dimension $d/2=1+\sqrt{2}$, while the $L$-orbits of invertible objects remain invertible in $\mathcal{C}_L$.  Therefore $\mathcal{C}_L$ is a braided generalized near-group fusion category with nontrivial dimension $d=1+\sqrt{2}$.  By the argument in the $k=2$ case above, $(\mathcal{C}_L)_\mathrm{ad}\simeq\mathcal{C}(\mathfrak{sl}_2,8,q)$ for a primitive $16$th root of unity $q$, and moreover $\mathcal{C}_L\simeq(\mathcal{C}_L)_\mathrm{ad}\boxtimes\mathcal{P}$ for a pointed braided fusion category $\mathcal{P}$ by Lemma \ref{biglem}.  In sum, if $\mathcal{C}$ exists, then $L$ acts on $\mathcal{C}_L$ by tensor autoequivalences and we may consider $\mathcal{C}$ as the equivariantization $\mathcal{C}\simeq((\mathcal{C}_L)_\mathrm{ad}\boxtimes\mathcal{P})^L$.  Any such tensor autoequivalence preserves (but may act nontrivially) on $(\mathcal{C}_L)_\mathrm{ad}$, therefore the same is true of $\mathcal{P}$.   Moreover $\mathcal{C}=(\mathcal{C}_L)^L\simeq((\mathcal{C}_L)_\mathrm{ad})^L\boxtimes(\mathcal{P})^L$.  The $L$-action restricted to $(\mathcal{C}_L)_\mathrm{ad}$ must permute the simple objects of dimension $d/2$ so that there exists a simple object of dimension $d$ in $\mathcal{C}$ \cite[Corollary 2.13]{MR3059899}.  But there is a unique nontrivial tensor autoequivalence of $(\mathcal{C}_L)_\mathrm{ad}\simeq\mathcal{C}(\mathfrak{sl}_2,8,q)$ as a fusion category, and $((\mathcal{C}_L)_\mathrm{ad})^L$ is the near-group fusion category with Grothendieck ring $R(C_2^2,4)$ \cite[Example 12.11]{MR1832764}.  This cannot occur as $R(C_2^2,4)$ has no braided categorifications \cite[Theorem III.4.6]{thornton2012generalized}.
\end{proof}

\begin{theorem}\label{thm:irr}
Let $\mathcal{C}$ be a braided generalized near-group fusion category with $\mathrm{FPdim}(\mathcal{C})\not\in\mathbb{Z}$.  Then $\mathcal{C}\simeq\mathcal{C}_\mathrm{ad}\boxtimes\mathcal{P}$ is a braided equivalence where $\mathcal{P}$ is a pointed braided fusion category and $\mathcal{C}_\mathrm{ad}$ is either $\mathcal{C}(\mathfrak{sl}_2,5,q)_\mathrm{ad}$ for a primitive 10th root of unity $q$ or $\mathcal{C}(\mathfrak{sl}_2,8,q)_\mathrm{ad}$ for a primitive $16$th root of unity $q$.
\end{theorem}

\begin{proof}
Let $\mathcal{C}$ be a generalized near-group fusion category such that $\mathrm{FPdim}(\mathcal{C})\not\in\mathbb{Z}$.  Then $(\mathcal{C}_\mathrm{ad})_\mathrm{ad}=\mathcal{C}_\mathrm{ad}$ by Lemma \ref{bigone} and therefore $(\mathcal{C}_\mathrm{ad})_\mathrm{pt}$ is symmetrically braided \cite[Proposition 3.13]{MR4167662}.  If $(\mathcal{C}_\mathrm{ad})_\mathrm{pt}$ is Tannakian, then Lemma \ref{tan} implies $\mathcal{C}_\mathrm{ad}\simeq\mathcal{C}(\mathfrak{sl}_2,5,q)_\mathrm{ad}$ for a primitive 10th root of unity $q$.  Each of these categories is nondegenerately braided and therefore $\mathcal{C}\simeq\mathcal{C}_\mathrm{ad}\boxtimes C_\mathcal{C}(\mathcal{C}_\mathrm{ad})$ \cite[Theorem 4.2]{mug1}.  Moreover $C_\mathcal{C}(\mathcal{C}_\mathrm{ad})$ is pointed otherwise $\mathcal{C}$ has simple objects with at least 3 distinct dimensions.  If $(\mathcal{C}_\mathrm{ad})_\mathrm{pt}$ is super-Tannakian, Lemma \ref{nottan} implies $\mathcal{C}_\mathrm{ad}\simeq\mathcal{C}(\mathfrak{sl}_2,8,q)_\mathrm{ad}$ is a braided equivalence for a primitive $16$th root of unity $q$, and our result follows from Lemma \ref{biglem}.
\end{proof}


\section{Epilogue: Nilpotency}\label{sec:nurrr}

Braided generalized near-group fusion categories $\mathcal{C}$ which are also nilpotent are the most tractable examples with integer Frobenius-Perron dimension.  Since $X\otimes X^\ast\cong Y\otimes Y^\ast$ for all simple $X,Y$ \cite[Lemma 2.8(d)]{220807319}, then $\mathcal{C}$ is nilpotent if and only if $\mathcal{C}_\mathrm{ad}$ is pointed, which by the transitivity of the action of invertible objects on noninvertible simple objects occurs if and only if $X\otimes Y$ is a sum of invertible objects for every noninvertible $X,Y\in\mathcal{O}(\mathcal{C})$.  These lie in the class of nilpotent fusion categories which have been called the braided generalized Tambara-Yamagami fusion categories in the past \cite[Section 5]{MR3071133}.   It suffices to understand those whose Frobenius-Perron dimension is a prime power by \cite[Theorem 1.1]{drinfeld2007grouptheoretical} which states that braided nilpotent fusion categories have a factorization into a product of braided fusion categories of prime power dimension.  Any such fusion category is $C_2$-graded, with all the noninvertible simple objects in the nontrivial component, which implies $\mathcal{C}$ is a product of a braided pointed fusion category, and a braided generalized near-group fusion category of dimension $2^n$ for some positive integer $n$ by \cite[Theorem 3.5.2]{tcat}.  The fact that $\dim(X)^2$ is always a power of $2$ in nilpotent braided generalized near-group fusion categories also follows from the following (cf.\ \cite[Theorem 1.2(1)]{siehler}).

\begin{proposition}\label{elem}
Let $\mathcal{C}$ be a $C_2$-graded extension of a pointed fusion category.  If $\mathcal{C}$ is braided, then $\mathcal{O}(\mathcal{C}_\mathrm{ad})$ is an elementary abelian $2$-group. 
\end{proposition}

\begin{proof}
Let $G:=\mathcal{O}(\mathcal{C}_\mathrm{pt})$ and $H:=\mathcal{O}(\mathcal{C}_\mathrm{ad})\subset G$ the fixed-point subgroup of the nontrivial component.  Such fusion categories $\mathcal{C}$ are characterized in \cite[Theorem 3.16]{MR3059669} where $G$ is denoted $S$, and $H$ is denoted $A$.  Almost all of the characterizing data is irrelevant to our argument except the symmetric nondegenerate bilinear form $\chi:H\times H\to\mathbb{C}^\times$.  This represents the associative isomorphisms $a_{g,X,h}=\chi(g,h)$ for any $g,h\in H$ and noninvertible $X\in\mathcal{O}(\mathcal{C})$ \cite[Corollary 3.11]{MR3059669}.  The associators $a_{g,X,h}$ and $a_{X,g,h}$ as well as the braiding isomorphisms $\tau_{g,h}$ and $\tau_{g,X}$ are nonzero scalars.  Fix $g,h\in H$.  The (first) hexagon axiom \cite[Definition 8.1]{tcat} for the triple of simple objects $g,X,h$ is
\begin{align}
&&\alpha_{g,X,h}\tau_{g,X}\alpha_{X,h,g}&=\tau_{g,X}\alpha_{X,g,h}\tau_{g,h} \\
\Rightarrow&&\chi(g,h)&=\dfrac{\alpha_{X,g,h}}{\alpha_{X,h,g}}\tau_{g,h},\label{dirty1}
\end{align}
and the (first) hexagon axiom for $h,X,g$ is
\begin{align}
&&\alpha_{h,X,g}\tau_{h,X}\alpha_{X,g,h}&=\tau_{h,X}\alpha_{X,h,g}\tau_{h,g} \\
\Rightarrow&&\chi(h,g)&=\dfrac{\alpha_{X,h,g}}{\alpha_{X,g,h}}\tau_{h,g}.\label{dirty2}
\end{align}
Multiplying (\ref{dirty1}) and (\ref{dirty2}) and using the symmetry of $\chi$ yields
\begin{equation}
\chi(g,h)^2=\chi(h,g)\chi(g,h)=\dfrac{\alpha_{X,h,g}}{\alpha_{X,g,h}}\tau_{h,g}\dfrac{\alpha_{X,g,h}}{\alpha_{X,h,g}}\tau_{g,h}=\tau_{h,g}\tau_{g,h}=1,
\end{equation}
since $g$ and $h$ centralize one another by \cite[Proposition 2.1]{MR4195420}.  A nondegenerate symmetric bilinear form taking only the values $\pm1$ only exists for elementary abelian $2$-groups (refer to \cite[Section 3.2]{siehler}, for example).
\end{proof}

\begin{example}
\par Copious examples can be constructed with fusion rules of $R(m,n)$ for $m,n\in\mathbb{Z}_{\geq1}$ (see Example \ref{ty}) from any braided Tambara-Yamagami fusion category, and and braided pointed fusion category on a cyclic $2$-group.  For a classical example, first consider $\mathrm{Rep}(Q_8)\boxtimes\mathrm{Rep}(C_4)$.  Let $\rho\in\mathrm{Rep}(Q_8)$ be simple and not invertible and $g\in\mathrm{Rep}(C_4)$ be invertible of order 4.  Then $\rho\boxtimes g$ generates a proper subcategory $\mathcal{D}$, consisting of $8$ invertible objects of order at most 2, $h\boxtimes e$ and $h\boxtimes g^2$ for all invertible $h$ in $\mathrm{Rep}(Q_8)$, and $2$ nonisomorphic noninvertible simple objects $\rho\boxtimes g$ and $\rho\boxtimes g^3$.  The fusion rules of $\mathcal{D}$ have no nontrivial factorization.  But $\mathcal{D}$ is symmetrically (and trivially) braided, so it must be equivalent as a fusion category to $\mathrm{Rep}(G)$ for some finite group $G$.  The only finite group $G$ of order 16 which is not a nontrivial direct product, with eight linear characters of order at most 2, and two $2$-dimensional irreducible characters is the central product $C_4\circ Q_8$.
\end{example}

\begin{conjecture}\label{con}
Let $\mathcal{C}$ be a braided generalized near-group fusion category.  If $\mathcal{C}$ is nilpotent, then the Grothendieck ring of $\mathcal{C}$ is isomorphic to $R(m,n)\times\mathbb{Z}K$ for some $m,n\in\mathbb{Z}_{\geq1}$ and finite abelian group $K$.
\end{conjecture}

\bibliographystyle{abbrv}
\bibliography{bib}

\end{document}